\documentclass{amsart}
\usepackage{stmaryrd}
\usepackage{silence}
\usepackage{upgreek}
\usepackage{amssymb}
\usepackage{faktor}
\usepackage[all,cmtip]{xy}
\usepackage{amsmath} 
\usepackage{mathrsfs}
\usepackage{tikz}
\usepackage{tikz-cd}
\usepackage{enumitem}
\usepackage{mathtools}
\usepackage[mathcal]{euscript}
\usepackage{graphicx}
\usepackage{url}
\usepackage{hyperref}
\usepackage[T1]{fontenc}
\usepackage{subfigure}

\usetikzlibrary{shapes.geometric}
\usetikzlibrary{decorations.markings}
\usetikzlibrary{patterns,decorations.pathreplacing}

\usepackage{ragged2e}

\newcommand{\id}{\mathrm{id}}

\DeclareMathOperator*{\colim}{\mathrm{colim}}

\definecolor{coloryellow}{RGB}{240,228,66}
\definecolor{colorskyblue}{RGB}{86,180,233}
\definecolor{colorvermillion}{RGB}{213,94,0}

\DeclareSymbolFont{sfletters}{OT1}{cmss}{m}{n}
\DeclareMathSymbol{\sTheta}{\mathord}{sfletters}{"02}

\newtheorem*{AEGT}{Analog Eilenberg--Ganea theorem}

\theoremstyle{definition}
\newtheorem{definition}{Definition}[section]

\newtheorem{example}[definition]{Example}

\newtheorem{construction}[definition]{Construction}

\newtheorem{problem}[definition]{Problem}

\theoremstyle{plain}
\newtheorem{proposition}[definition]{Proposition}
\newtheorem{lemma}[definition]{Lemma}
\newtheorem{corollary}[definition]{Corollary}
\newtheorem{theorem}[definition]{Theorem}
\newtheorem{conjecture}[definition]{Conjecture}

\theoremstyle{remark}
\newtheorem{remark}[definition]{Remark}

\makeatletter
\@addtoreset{definition}{section}
\makeatother

\usepackage{marginnote}
    \DeclareFontFamily{U}{wncy}{}
    \DeclareFontShape{U}{wncy}{m}{n}{<->wncyr10}{}
    \DeclareSymbolFont{mcy}{U}{wncy}{m}{n}
    \DeclareMathSymbol{\Sha}{\mathord}{mcy}{"58}

\newsavebox{\foobox}

\title{Analog category and complexity}
\author{Ben Knudsen and Shmuel Weinberger}

\begin{document}

\maketitle

\begin{abstract}
We study probabilistic variants of the Lusternik--Schnirelmann category and topological complexity, which bound the classical invariants from below. We present a number of computations illustrating both wide agreement and wide disagreement with the classical notions. In the aspherical case, where our invariants are group invariants, we establish a counterpart of the Eilenberg--Ganea theorem in the torsion-free case, as well as a contrasting universal upper bound in the finite case.
\end{abstract}

\section{Introduction}

In 1987, Smale \cite{Smale:TA} launched a new approach to computational complexity based on continuous mathematics, measuring the number of paths of computation that a ``real number'' machine requires to find a polynomial root. Smale showed that a certain amount of discontinuity is inevitable in any method of root finding by bounding the Schwarz genus of the map from ordered to unordered configurations in $\mathbb{C}$ using cohomology.

Building on Smale's idea, Farber \cite{Farber:TCMP} initiated the study of the topological complexity of motion planning, a fundamental problem in robotics. As explained in \cite{Farber:TRMP}, this notion also admits an interpretation as a measure of the minimal level of \emph{randomness} in a motion planning algorithm. This model of randomness is ``digital'': given a starting and ending point, the algorithm implements the motion labeled by an element of $\{1,2,\ldots, k\}$ chosen randomly according to a fixed distribution.

In this paper, we allow algorithms with ``analog'' randomness. Consider the round metric on $\mathbb{RP}^d$, in which every pair of points is joined by either one or two shortest geodesics, and, when there is only one, the second shortest is the geodesic in the reverse direction. We may choose between these two options according to a distribution that interpolates continuously between the first case, where the two options are distinguishable by length, and the second, where they are not (see Proposition \ref{prop:projective}). Thus, with analog randomness, a robot on $\mathbb{RP}^d$ needs only \emph{two rules}, a sharp contrast to the beautiful result of \cite{FarberTabachnikovYuzvinsky:TRMPPS} calculating the ``digital'' topological complexity of $\mathbb{RP}^d$ as its immersion dimension, which grows linearly in $d$.

More formally, we define a notion of \emph{analog sectional category} (Definition \ref{def:asc}) of a map by recording the minimal support of a continuous choice of fiberwise probability measure. Specializing to path fibrations, we obtain invariants $\mathrm{acat}(X)$ and $\mathrm{ATC}_r(X)$ (Definition \ref{def:atc}), analog counterparts of the Lusternik--Schnirelmann category \cite{LusternikSchnirelmann:MTPV} and $r$th sequential topological complexity \cite{Rudyak:HATC}. Our main computational result (Theorem \ref{thm:cd} below) is the following.

\begin{AEGT}
For any torsion-free group $G$, we have the equality $\mathrm{acat}(BG)=\mathrm{cd}(G).$
\end{AEGT}

In a striking departure from the classical setting, this equality fails for finite groups; indeed, as shown in Theorem \ref{thm:finite group}, the quantities $\mathrm{ATC}_r(BG)$ are all bounded above by $|G|-1$ for $G$ finite. 

\begin{problem}\label{problem:algebra}
Describe the group invariants $\mathrm{ATC}_r(BG)$ algebraically.
\end{problem}

The classical counterpart of this problem is regarded as perhaps the most difficult open problem in the study of (sequential) topological complexity \cite{GrantLuptonOprea:NLBTCAS,FarberOprea:HTCAS, FarberMescher:TCASS, EspinosaFarberMescherOprea:STCASSCSI}, and even calculations for specific groups are often difficult \cite{CohenVandembroucq:TCKB, Knudsen:TCPGBGSM,AguilarGuzmanGonzalezOprea:RAAGPPTCGF,Knudsen:OSTCGBG}. Nevertheless, given how much simpler some analog calculations are (projective spaces, for example), there is reason to hope that Problem \ref{problem:algebra} might be more tractable. If so, it may even shed light on the classical problem, since the analog invariants bound the digital invariants from below.

In the final stages of this paper's completion, there appeared the preprint \cite{DranishnikovJauhari:DTCLSC}, which introduces the notions of ``distributional'' category and topological complexity. Although the motivations differ, the basic idea seems very similar to ours. We welcome their paper as evidence of the naturality of these ideas, and we recommend it to the interested reader for a different selection of first calculations. We expect the definitions to be very closely related.

\begin{conjecture}\label{conj:coincidence}
The analog and distributional notions of category and topological complexity coincide on the class of metrizable spaces.
\end{conjecture}

The relevance of metrizability is that the distributional invariants are defined using the L\'{e}vy--Prokhorov metric, requiring that the background space itself be metrizable, which is avoided in our approach---see Remark \ref{remark:metrizable} below. Consequently, the analog invariants are defined on a much larger class of spaces. This technicality becomes relevant when studying the particularly interesting class of aspherical spaces---see Remark \ref{remark:locally finite} below.

\subsection{Conventions} We use the following slightly non-standard notational convention for topological simplices:
\[\Delta^{n-1}:=\left\{(t_1,\ldots, t_n)\in[0,1]^n:\sum_{i=1}^nt_i=1\right\}.\] We work in Steenrod's convenient category of topological spaces \cite{Steenrod:CCTS}---see Appendix \ref{appendix:convenient} for a summary of relevant facts about these spaces. Topological spaces are implicitly convenient, as are limits, including products, and mapping spaces. Convenient colimits, when they exist, are the same as ordinary colimits. The adjective ``compact'' refers to the definition in terms of open covers.

\subsection{Acknowledgements} The first author thanks Matan Harel for helpful conversations. The paper benefited from comments from John Oprea and two anonymous referees.

\section{Probability measures with finite support}

Intuitively, a probability measure with finite support is a convex linear combination $\sum_{i=1}^nt_i\delta_{x_i}$, where $\delta_x$ is the point mass or Dirac delta ``function'' assigning full measure to the singleton $\{x\}$. Of course, measure theory is one natural habitat for such an object, but a purely combinatorial and topological approach is most appropriate for our purposes.

\begin{definition}\label{def:probability measures}
Let $X$ be a topological space. The space of \emph{probability measures with finite support} on $X$ is the quotient space 
\[
\mathcal{P}(X)=\faktor{\bigsqcup_{n>0}X^n\times \Delta^{n-1}}{\sim}
\] where the indicated equivalence relation is generated by the following three relations:
\begin{enumerate}
\item $(x_1,\ldots, x_n,t_1,\ldots, t_n)\sim(x_{\sigma(1)},\ldots, x_{\sigma(n)}, t_{\sigma(1)},\ldots, t_{\sigma(n)})$ for $\sigma\in \Sigma_n$;
\item $(x_1,\ldots, x_n,t_1,\ldots, t_n)\sim (x_1,\ldots, x_{n-1}, t_1,\ldots, t_{n-1}+t_n)$ if $x_{n-1}=x_n$; and
\item $(x_1,\ldots, x_n,t_1,\ldots, t_n)\sim (x_1,\ldots, x_{n-1}, t_1,\ldots, t_{n-1})$ if $t_n=0$.
\end{enumerate}
\end{definition}

\begin{remark}\label{remark:metrizable}
When $X$ is metrizable, more classical topologies are available on the set $\mathcal{P}(X)$, e.g., the metric topologies induced by the L\'{e}vy--Prokhorov and Wasserstein metrics. We adopt this purely topological definition for two reasons: first, as articulated in Theorem \ref{thm:point set}, it carries numerous desirable technical properties; second, our examples of greatest interest are typically non-metrizable, at least on the nose (see Remark \ref{remark:locally finite}).

As a referee informed us, a topology on $\mathcal{P}_n(X)$ for $n$ fixed was previously defined by Fedorchuk \cite{Fedorchuk:PMANR} (see also \cite{Cu:PMFSTS}). Although we have not checked the details, it seems likely that this topology coincides with the quotient topology in the compact case but not in general (see Lemma \ref{lem:big basis} below). The quotient topology on $\mathcal{P}_n(X)$ was previously studied by Kallel--Karoui \cite{KallelKaroui:SJWB}.
\end{remark}

According to the motivation indicated above, we denote the equivalence class of $(x_1,\ldots, x_n, t_1,\ldots, t_n)$ by $\sum_{i=1}^nt_i\delta_{x_i}$. 

\begin{lemma}\label{lem:lowest terms}
Let $X$ be a topological space. Every point of $\mathcal{P}(X)$ can be written uniquely in the form $\sum_{i=1}^nt_i\delta_{x_i}$ subject to the following two conditions:
\begin{enumerate}
\item $t_i>0$ for $1\leq i\leq n$; and
\item $x_i\neq x_j$ for $i\neq j$.
\end{enumerate}
\end{lemma}

Thus, written in lowest terms, the coefficient of $\delta_x$ in $\mu\in \mathcal{P}(X)$ is well-defined and denoted $\mu(x)$.

\begin{definition}\label{def:support}
Let $X$ be a topological space. The \emph{support} of $\mu\in\mathcal{P}(X)$ is the (finite) subset $\mathrm{supp}(\mu)=\{x\in X: \mu(x)\neq0\}$. We write $\mathcal{P}_n(X)\subseteq \mathcal{P}(X)$ for the subspace of probability measures $\mu$ with $|\mathrm{supp}(\mu)|\leq n$. 
\end{definition}

The definitions imply directly that $\mathcal{P}_n(X)\subseteq \mathcal{P}(X)$ is a closed subspace.

\begin{example}
The assignment $x\mapsto \delta_x$ determines a continuous map $\eta_X:X\to \mathcal{P}(X)$ restricting to a homeomorphism $X\cong \mathcal{P}_1(X)$.
\end{example}

\begin{example}\label{example:discrete}
If $X$ is discrete, then the assignment $\mu\mapsto \sum_{x\in X} \mu(x)x$ determines a homeomorphism $\mathcal{P}(X)\cong\Delta^X$, the simplex with vertex set $X$. This homeomorphism identifies $\mathcal{P}_n(X)$ with the $(n-1)$-skeleton of $\Delta^X$. Continuity is immediate from properties of the quotient topology, and bijectivity is an easy exercise. When $X$ is finite, the source is compact and the target Hausdorff, so the claim follows; in the infinite case, both spaces carry the colimit topology inherited from the finite subsets of $X$.
\end{example}

The construction $X\mapsto \mathcal{P}(X)$ extends canonically to an endofunctor on the category of convenient spaces; given a continuous map $f:X\to Y$, the maps $f^n\times\Delta^{n-1}$ for varying $n$ descend to a continuous map $f_*:\mathcal{P}(X)\to \mathcal{P}(Y)$. Concretely, we have
\[
f_*\left(\sum_{i=1}^nt_i\delta_{x_i}\right)=\sum_{i=1}^nt_i\delta_{f(x_i)}.
\]

Concerning this functor, we have the following fundamental result, whose proof is deferred to Section \ref{section:proof of point set}.

\begin{theorem}\label{thm:point set}
The functor $\mathcal{P}$ is an endofunctor on the category of convenient spaces, which preserves homotopy, sifted colimits, quotient maps, and closed embeddings.
\end{theorem}

In particular, the map $f_*$ is a homotopy equivalence whenever $f$ is so.

\section{Extended functoriality}

In addition to the functoriality articulated in the previous section, the construction $\mathcal{P}$ carries several pieces of algebraic structure that will be indispensable in our later investigation. 

\begin{construction}
We define a function $\kappa_X:\mathcal{P}(\mathcal{P}(X))\to \mathcal{P}(X)$ as follows. Given $\mu_j\in\mathcal{P}(X)$ for $1\leq j\leq k$, write $\mu_j=\sum_{i=1}^{n_j}t_{ij}\delta_{x_{ij}}$, and define 
\[\kappa_X\left(\sum_{j=1}^ks_j\delta_{\mu_j}\right)=\sum_{j=1}^k\sum_{i=1}^{n_j}s_jt_{ij}\delta_{x_{ij}}.
\]
\end{construction}

It is easy to check that $\kappa_X$ is well-defined. Notice that $\kappa_X(\mathcal{P}_k(\mathcal{P}_n(X)))\subseteq \mathcal{P}_{kn}(X)$ by construction.

\begin{proposition}\label{prop:monad}
The triple $(\mathcal{P},\kappa,\eta)$ is a monad.
\end{proposition}
\begin{proof}
The main point is to verify the continuity of $\kappa_X$ for $X$ fixed. We begin by noting that the obvious map \[\bigsqcup_{k>0}\bigsqcup_{n>0}  \bigsqcup_{n_1+\cdots+n_k=n}\left(\bigsqcap_{j=1}^kX^{n_j}\times\Delta^{n_j-1}\right)\times\Delta^{k-1}\to\mathcal{P}(\mathcal{P}(X))\] is a quotient map (we use Proposition \ref{prop:quotient product}). Thus, continuity follows from commutativity of the diagram
{\small\[\xymatrix{
\displaystyle\left(\bigsqcap_{j=1}^k X^{n_j}\times\Delta^{n_j-1}\right)\times\Delta^{k-1}\ar[d]\ar@{=}[r]^-\sim& \displaystyle X^{n}\times\left(\bigsqcap_{j=1}^k\Delta^{n_j-1}\right)\times\Delta^{k-1}\ar[r]&X^n\times\Delta^{n-1}\ar[d]\\
\mathcal{P}(\mathcal{P}(X))\ar[rr]^-{\kappa_X}&&\mathcal{P}(X),
}\]}where the homeomorphism is the canonical one and the upper righthand horizontal map is obtained by restricting a product of scaling maps on various Euclidean spaces, hence continuous.

Naturality of $\eta$ and $\kappa$, and the two monad identities $\kappa_{X}\circ(\mu_{X})_*=\kappa_{X}\circ\kappa_{\mathcal{P}(X)}$ and $\kappa_{X}\circ(\eta_X)_*=\id_{\mathcal{P}(X)}=\kappa_{X}\circ\eta_{\mathcal{P}(X)}$, are essentially immediate from the definitions, after recalling that $\eta_X$ sends $x$ to the equivalence class of $(x,1)$.
\end{proof}

Using this structure, we obtain a certain limited contravariant functoriality.

\begin{proposition}\label{prop:contravariant}
Let $p:E\to X$ be a degree $k$ covering map. The assignment \[\mu\mapsto \frac{1}{k}\sum_{x\in \mathrm{supp}(\mu)}\sum_{\tilde x\in p^{-1}(x)} \mu(x)\delta_{\tilde x}\] defines a continuous map $p^*:\mathcal{P}(X)\to \mathcal{P}(E)$, which is a section of $p_*$.
\end{proposition}

\begin{lemma}\label{lem:symmetric continuity}
Let $p:E\to X$ be a degree $k$ covering map. The assignment $x\mapsto p^{-1}(x)$ defines a continuous map $p^{-1}:X\to \mathrm{Sym}^k(E)$, where $\mathrm{Sym}^k(E):=E^k/\Sigma_k$ is the $k$th symmetric power.
\end{lemma}
\begin{proof}
Assume first that $X$ is compact, so that $E$ is as well. In this case, $\mathrm{Sym}^k(E)$ carries the quotient topology from the ordinary Cartesian power $E^k$. Fix $x_0\in X$, and choose an open neighborhood $p^{-1}(x_0)\in  U\subseteq \mathrm{Sym}^k(E)$. Choosing an ordering $p^{-1}(x_0)=\{\tilde x_i\}_{i=1}^k$, continuity of the projection $\pi:E^k\to \mathrm{Sym}^k(E)$ grants open neighborhoods $x_i\in U_i\subseteq E$ such that $\pi(\sqcap_{i=1}^k U_i)\subseteq U$. Since $p$ is a covering map, we may assume without loss of generality that the $U_i$ are pairwise disjoint, that $V:=p(U_i)$ is independent of $i$, and that $p|_{U_i}$ is a homeomorphism onto $V$ for each $i$. Since $V$ is open, and since the function in question maps $V$ into $U$, continuity follows.

In the general case, writing $\mathcal{K}$ for the set of compact subsets of $X$, ordered by inclusion, consider the following commutative diagram of continuous maps:
\[
\xymatrix{
\colim_{K\in\mathcal{K}}K\ar[d]\ar[rr]^-{\colim_{K\in\mathcal{K}} p^{-1}}&&\colim_{K\in\mathcal{K}}\mathrm{Sym}^k(p^{-1}(K))\ar[d]\\
X\ar@{-->}[rr]&&\mathrm{Sym}^k(E)
}
\] Since $X$ is covenient, the lefthand vertical arrow is a homeomorphism. To see that the righthand vertical arrow is so as well, we note that $\mathrm{Sym}^k$ preserves sifted, hence filtered, colimits by Proposition \ref{prop:cartesian}, and that the natural map $\colim_{K\in\mathcal{K}}p^{-1}(K)\to E$ is a homeomorphism, since $E$ is convenient and any compact subset of $E$ lies in one of the form $p^{-1}(K)$. Thus, the continuous dashed filler is determined, and it is easy to check that it coincides with $p^{-1}$ as a set function.
\end{proof}

\begin{proof}[Proof of Proposition \ref{prop:contravariant}] One checks easily that the function $p^*$ coincides with the composite
\[\mathcal{P}(X)\xrightarrow{(p^{-1})_*} \mathcal{P}(\mathrm{Sym}^k(E))\to \mathcal{P}(E^k\times_{\Sigma_k}\Delta^{k-1})\to\mathcal{P}(\mathcal{P}(E))\xrightarrow{\kappa_E} \mathcal{P}(E),\] where the second arrow is induced by the inclusion of the barycenter with Cartesian coordinates $(e_1,\ldots, e_k,\frac{1}{k},\ldots, \frac{1}{k})$. The first and last arrows are continuous by Lemma \ref{lem:symmetric continuity} and Proposition \ref{prop:monad}, respectively, and the second and third are values of the functor $\mathcal{P}$ on continuous maps. Continuity follows, and the identity $p_*\circ p^*=\id_{\mathcal{P}(X)}$ is immediate from the definition.
\end{proof}

We note that, in the setting of Proposition \ref{prop:contravariant}, we have the inclusion $p^*(\mathcal{P}_n(B))\subseteq \mathcal{P}_{kn}(E)$.

There is also a kind of external product on measures. 

\begin{proposition}\label{prop:external product}
Let $X$ and $Y$ be topological spaces. The assignment 
\[
\left(\sum_{i=1}^nt_i\delta_{x_i},\sum_{j=1}^ps_j\delta_{y_j}\right)\mapsto \sum_{i=1}^n\sum_{j=1}^pt_is_j\delta_{(x_i,y_j)}
\] defines a continuous map $\boxtimes:\mathcal{P}(X)\times\mathcal{P}(Y)\to \mathcal{P}(X\times Y)$, which admits a canonical retraction.
\end{proposition}
\begin{proof}
Continuity is essentially immediate from properties of the quotient topology, together with Proposition \ref{prop:quotient product}, and it is easily checked that a retraction is given by the map $\mathcal{P}(X\times Y)\xrightarrow{} \mathcal{P}(X)\times\mathcal{P}(Y)$ induced by the projections onto $X$ and $Y$.
\end{proof}

\section{Relative probability measures}

The invariants we wish to define are special cases of a general invariant of maps, for which we require a relativization of the preceding ideas.

\begin{definition}
Let $f:X\to Y$ be a continuous map with $X$ convenient and $Y$ Hausdorff. The space of \emph{$f$-relative measures of finite support} on $X$ is the subspace 
\[\mathcal{P}(f)=\{\mu\in \mathcal{P}(X): |f(\mathrm{supp}(\mu))|=1\}.\]
\end{definition}

Although we do not require that $Y$ be convenient in this definition, we may always assume so; indeed, if $\tilde f: X\to k(Y)$ is the adjunct of $f$, then $\mathcal{P}(f)=\mathcal{P}(\tilde f)$ as subspaces of $\mathcal{P}(X)$.

\begin{lemma}\label{lem:relative closed}
For any $f:X\to Y$, the subspace $\mathcal{P}(f)\subseteq \mathcal{P}(X)$ is closed.
\end{lemma}
\begin{proof}
Without loss of generality, we may take the target of $f$ to be convenient, in which case $\mathcal{P}(f)$ is the preimage under $f_*$ of the closed set $\mathcal{P}_1(Y)$.
\end{proof}

In particular, it follows that $\mathcal{P}(f)$ is always convenient. Our previous discussion of functoriality implies immediately that the construction $f\mapsto\mathcal{P}(f)$ extends to a functor on the arrow category of convenient space, which sends relative homotopy to relative homotopy. Likewise, there is the stratification $\mathcal{P}_n(f):=\mathcal{P}_n(X)\cap \mathcal{P}(f)$. Note, moreover, that $\mathcal{P}_1(f)=\mathcal{P}_1(X)$.

\begin{lemma}
Let $f:X\to Y$ be a continuous map. The assignment $\mu\mapsto f(\mathrm{supp}(\mu))$ defines a continuous map $p:\mathcal{P}(f)\to Y$.
\end{lemma}
\begin{proof}
From the definition of $\mathcal{P}(f)$ and the fact that $\eta_Y$ is a topological embedding, the dashed filler exists in the following commuting diagram of continuous maps:
\[\xymatrix{
\mathcal{P}(f)\ar[r]^-\subseteq\ar@{-->}[d]&\mathcal{P}(X)\ar[d]^-{f_*}\\
Y\ar[r]^-{\eta_Y}&\mathcal{P}(Y)
}\] Since this dashed filler coincides with the function in question, the claim follows.
\end{proof}

This map, which we refer to generically as ``the projection'' is of fundamental importance to us. We note that a pair of maps $f_i:X_i\to Y$ and a map $g:X_1\to X_2$ with $f_2\circ g=f_1$ induce a commutative diagram 
\[\xymatrix{
\mathcal{P}(f_1)\ar[dr]_-{p_{f_1}}\ar[rr]^-{g_*}&&\mathcal{P}(f_2)\ar[dl]^-{p_{f_2}}\\
&Y
}
\] Using this structure, we may formulate a relative external product of measures.

\begin{proposition}\label{prop:pullback}
Let $f_i:X_i\to Y$ be continuous maps and $g:X_1\times_Y X_2 \to Y$ the canonical map from the pullback. The formula of Proposition \ref{prop:external product} defines a continuous map $\boxtimes_Y:\mathcal{P}(f_1)\times_Y\mathcal{P}(f_2)\to \mathcal{P}(g),$ which admits a canonical retraction.
\end{proposition}
\begin{proof}
Consider the following commutative diagram of continuous maps:
\[
\xymatrix{
\mathcal{P}(f_1)\times_Y\mathcal{P}(f_2)\ar[d]_-\subseteq\ar@{-->}[r]^-{\boxtimes_Y}&\mathcal{P}(g)\ar[d]^-\subseteq\\
\mathcal{P}(f_1)\times\mathcal{P}(f_2)\ar[d]&\mathcal{P}(X_1\times_Y X_2)\ar[d]\\
\mathcal{P}(X_1)\times\mathcal{P}(X_2)\ar[r]^-{\boxtimes}&\mathcal{P}(X_1\times X_2).
}
\] As a product of embeddings, the bottom left map is also an embedding, and the bottom right map is an embedding by Theorem \ref{thm:point set} and Corollary \ref{cor:pullback inclusion}. Thus, the dashed filler, which is easily checked to be well-defined, is continuous. Similarly, it is easily checked that the retraction of $\boxtimes$ constructed in Proposition \ref{prop:external product} maps the image of the righthand vertical composite into the image of the lefthand vertical composite, hence restricts to a retraction of $\boxtimes_Y$.
\end{proof}

\begin{corollary}\label{cor:base change}
Let $f:X\to Y$ and $i:A\to Y$ be continuous maps with $i$ injective. There is a canonical natural homeomorphism $A\times_Y\mathcal{P}(f)\cong \mathcal{P}(f_A)$ over $A$, where $f_A:A\times_Y X\to A$ is the canonical map.
\end{corollary}
\begin{proof}
The assumption of injectivity implies that $\mathcal{P}(i)=\mathcal{P}_1(A)$ as subspaces of $\mathcal{P}(A)$ and that $\mathcal{P}(g)=\mathcal{P}(f_A)$ as subspaces of $\mathcal{P}(A\times_Y X)$. Thus, it suffices to establish that the map $\boxtimes_Y$ of Proposition \ref{prop:pullback} is bijective in this case, which again follows easily from injectivity of $i$.
\end{proof}

\begin{corollary}\label{cor:product decomposition}
Let $\pi_i:Y\times Z\to Y$ denote the projection onto the $i$th factor. There is a canonical natural homeomorphism $\mathcal{P}(\pi_1)\cong Y\times \mathcal{P}(Z)$ over $Y$.
\end{corollary}
\begin{proof}
It is easy to check that $\boxtimes: \mathcal{P}(Y)\times\mathcal{P}(Z)\to \mathcal{P}(Y\times Z)$ maps $\mathcal{P}_1(Y)\times\mathcal{P}(Z)$ bijectively onto $\mathcal{P}(\pi_1)$, so the claim follows from Proposition \ref{prop:external product} and fact that $\eta_Y$ is a homeomorphism.
\end{proof}

Using these results, we will be able to identify $\mathcal{P}(f)$ as a bundle in many cases.

\begin{definition}
Let $f:X\to Y$ be a continuous map. We say that $f$ is a \emph{convenient fiber bundle} with fiber $F$ if there is an open cover $\{U_\alpha\}$ of $Y$ and a collection of homeomorphisms $U_\alpha\times_Y X\cong U_\alpha\times F$ over $U_\alpha$.
\end{definition}

In this definition, the spaces $X$, $Y$, and $F$ are required to be convenient, as are the open neighborhoods $U_\alpha$ (note that convenience is not automatically inherited by open subsets). The indicated limits are, as always, the convenient limits. The language of transition functions, structure groups, and cocycles translates unchanged to the convenient setting.

Combining Corollaries \ref{cor:base change} and \ref{cor:product decomposition}, we conclude the following.

\begin{corollary}\label{cor:fiber bundle}
If $f:X\to Y$ is a convenient fiber bundle with fiber $F$ and discrete structure group $G$, then $p:\mathcal{P}(f)\to Y$ is a fiber bundle with fiber $\mathcal{P}(F)$ and structure group $G$. More specifically, if $\{g_{\alpha\beta}\}$ is a $G$-cocycle for $f$ defined on the trivializing cover $\{U_\alpha\}$, then it is also a $G$-cocycle for $p$.
\end{corollary}

\begin{remark}
We restrict to discrete $G$ only in order to avoid the necessity of establishing the promotion of $\mathcal{P}$ to a topologically enriched functor.
\end{remark}

In many cases of interest, convenient fiber bundles are simply fiber bundles in the ordinary sense.

\begin{lemma}\label{lem:bundle comparison}
In either of the following situations, a map $f:X\to Y$ is a convenient fiber bundle with fiber $F$ if and only if it is a fiber bundle with fiber $F$:
\begin{enumerate}
\item if $Y$ is a CW complex and $F$ is locally compact; or
\item if $Y$ is locally compact.
\end{enumerate}
\end{lemma}
\begin{proof}
In the first case, the assumption on $Y$ guarantees that every point has a local basis of contractible, convenient open neighborhoods; in the second case, local compactness guarantees that every open subset of $Y$ is convenient. In both cases, local compactness guarantees that the convenient and ordinary product and pullback coincide. 
\end{proof}

\section{Analog sectional category}

We now define a general invariant of maps, of which the analog category and analog topological complexity will be special cases.

\begin{definition}\label{def:asc}
Let $f:X\to Y$ be a continuous map. The \emph{analog sectional category} of $f$, denoted $\mathrm{asecat}(f)$, is the least $n$ such that the projection $\mathcal{P}(f)\to Y$ admits a section with image lying in $\mathcal{P}_{n+1}(f)$.
\end{definition}

As the following result shows, the analog sectional category is a homotopy invariant in an appropriate sense.

\begin{proposition}\label{prop:fibration}
Consider the following commutative diagram of continuous maps:
\[\xymatrix{
X'\ar[d]_-{f'}\ar[r]^-{\tilde h}&X\ar[d]^-{f}\\
Y'\ar[r]^-h&Y
}\] If the vertical maps are fibrations and the horizontal maps homotopy equivalences, then $\mathrm{asecat}(f)=\mathrm{asecat}(f')$.
\end{proposition}
\begin{proof}
Our assumptions imply that the natural map $X'\to X\times_Y Y'$ is a fiber homotopy equivalence over $Y'$, so we may choose a fiber homotopy inverse $k$. Writing $s:Y\to \mathcal{P}(f)$ for a section with image lying in $\mathcal{P}_{n+1}(f)$, consider the composite 
\[Y'\xrightarrow{(s\circ h,\id)}\mathcal{P}(f)\times_Y Y'\xrightarrow{\id\times\eta_{Y'}}\mathcal{P}(f)\times_Y\mathcal{P}(h)\xrightarrow{\boxtimes_Y}\mathcal{P}(g),\] where $g:X\times_Y Y'\to Y$ is the natural map. By inspection, this composite factors through the subspace $\mathcal{P}(g')\subseteq \mathcal{P}(g)$, where $g':X\times_Y Y'\to Y'$ is projection onto the second factor. Postcomposing with $k_*$ then yields a section $Y'\to \mathcal{P}(f')$ with image lying in $\mathcal{P}_{n+1}(f')$, and it follows that $\mathrm{asecat}(f')\leq \mathrm{asecat}(f)$. After replacing $h$ and $\tilde h$ by their homotopy inverses and using our assumption on $f$ to strictify the resulting homotopy commutative diagram, the reverse inequality follows from the same argument.
\end{proof}

Since replacements of maps by fibrations are unique up to fiberwise homotopy equivalence, we conclude that what one might call the \emph{derived} analog sectional category is well-defined.

\begin{corollary}\label{cor:derived}
Let $f:X\to Y$ be a map and $\tilde f$ a fibration weakly equivalent to $f$ in the arrow category. The quantity $\mathrm{asecat}(\tilde f)$ is independent of the choice of $\tilde f$.
\end{corollary}

It is easy to see that the analog sectional category of $f$ is $0$ if and only if $f$ admits a continuous section. More generally, we have the following.

\begin{proposition}\label{prop:cover}
Let $f:X\to Y$ be a continuous map. If $Y$ has an open cover $\mathcal{U}=\{U_i\}_{i=1}^{n+1}$ such that 
\begin{enumerate}
\item $f|_{f^{-1}(U_i)}$ admits a continuous section for every $1\leq i\leq n+1$, and
\item $Y$ admits a partition of unity subordinate to $\mathcal{U}$,
\end{enumerate}
then $\mathrm{asecat}(f)\leq n$.
\end{proposition}
\begin{proof}
Let $\{\varphi_i\}_{i=1}^{n+1}$ be a partition of unity subordinate to $\mathcal{U}$ and $s_i:U_i\to X$ a continuous section of $f|_{f^{-1}(U_i)}$ for $1\leq i\leq n+1$. We will show that the function $s:Y\to \mathcal{P}(X)$ given by \[s(y)=\sum_{i=1}^{n+1}\varphi_i(y)\delta_{s_i(y)}\] is continuous. Since the $s_i$ are sections, the image of $s$ lies  in $\mathcal{P}(f)$, and it lies in $\mathcal{P}_{n+1}(X)$ by inspection, so the claim will follow. Note that the condition that $\sum_{i=1}^{n+1} \varphi_i\equiv 1$ is necessary in order that $s$ be well-defined.

Writing $C(X)$ for the topological cone on $X$, define $\tilde s_i:Y\to C(X)$ by the formula 
\[\tilde s_i(y)=\begin{cases}
(s_i(y),\varphi_i(y))&\qquad y\in U_i\\
*&\qquad y \in Y\setminus \overline{\varphi_i^{-1}(\mathbb{R}_{>0})}
\end{cases}
\] It is easy to check that this function is well-defined, and it is continuous on each of its (open) domains of definition, hence continuous. These maps determine the bottom map in the following diagram of continuous maps:
\[\xymatrix{
&&X^{n+1}\times\Delta^{n}\ar[d]\ar[dl]\\
&X^{*(n+1)}\ar[d]\ar@{-->}[r]&\mathcal{P}(X)\\
Y\ar@{-->}[ur]\ar[r]&C(X)^{n+1}.
}\] Here, we have written $X^{*(n+1)}$ for the iterated join of $X$, which carries the quotient topology from the solid diagonal map, and which embeds in $C(X)^{n+1}$ as the subspace where the cone coordinates sum to $1$ (we use our standing assumption that $X$ is convenient). The dashed fillers are well-defined as functions, hence continuous by general properties of the subspace and quotient topologies. Since the composite of the dashed arrows coincides with $s$ as set functions, the claim follows.
\end{proof}

\begin{corollary}\label{cor:paracompact category}
Let $f:X\to Y$ be a continuous map. If $Y$ is paracompact, then $\mathrm{asecat}(f)\leq \mathrm{secat}(f)$.
\end{corollary}

\begin{remark}
Alternatively, Corollary \ref{cor:paracompact category} follows from \cite[Prop. 8.1]{James:CSLS}, together with the observation that $\mathcal{P}_{n}(f)$ is a quotient (over $Y$) of the $n$-fold fiberwise join of $f$. Of course, James' argument is no different from that of Proposition \ref{prop:cover}.
\end{remark}

\section{Analog category and analog topological complexity}

We are now able to define our invariants. In the following definition, we remind the reader of our standing convention that mapping spaces, topologized as in \cite[\S 5]{Steenrod:CCTS}, are convenient.

\begin{definition}\label{def:atc}
Let $X$ be a path connected topological space. 
\begin{enumerate}
\item The \emph{analog category} of $X$, denoted $\mathrm{acat}(X)$ or $\mathrm{ATC}_1(X)$, is the analog sectional category of the map \[\pi_X=\pi_X^1:(X,x_0)^{([0,1],\{0\})}\to X\] given by evaluation at $1$, where $x_0\in X$ is any basepoint.
\item For $r>1$ the $r$th \emph{sequential analog topological complexity} of $X$, denoted $\mathrm{ATC}_r(X),$ is the analog sectional category of the map \[\pi^r_X:X^{[0,1]}\to X^r\] given by evaluation at $\frac{i}{r-1}$ for $0\leq i<r$. In the case $r=2$, we write simply $\mathrm{ATC}(X)$ and refer to \emph{analog topological complexity}.
\end{enumerate}
\end{definition}

Note that, by fiberwise homotopy invariance, the analog category is independent of the choice of basepoint. 

\begin{remark}\label{remark:diagonal}
Equivalently, by Corollary \ref{cor:derived}, $\mathrm{acat}(X)$ is the analog sectional category of any replacement of the inclusion of the basepoint by a fibration, and $\mathrm{ATC}_r(X)$ is the analog sectional category of any replacement of the diagonal $X\to X^r$.
\end{remark}

We record a few basic properties of these invariants.

\begin{proposition}\label{prop:basic}
Let $X$ be a topological space.
\begin{enumerate}
\item If $X\simeq Y$, then $\mathrm{ATC}_r(X)=\mathrm{ATC}_r(Y)$ for every $r>0$.
\item The equality $\mathrm{ATC}_r(X)=0$ holds for some $r>0$ if and only if it holds for every $r>0$ if and only if $X$ is contractible.
\item If $X$ is paracompact, then $\mathrm{ATC}_r(X)\leq \mathrm{TC}_r(X)$ for every $r>0$.
\item The inequality $\mathrm{ATC}_r(X)\leq \mathrm{acat}(X^r)$ holds.
\item The quantity $\mathrm{ATC}_r(X)$ is non-decreasing in $r$.
\end{enumerate}
\end{proposition}
\begin{proof}
The first and third claims are immediate from Propositions \ref{prop:fibration} and \ref{prop:cover}, respectively. In light of Remark \ref{remark:diagonal}, the fourth and fifth claims follow from Corollary \ref{cor:derived} and the commuting diagram 
\[
\xymatrix{
\Delta^0\ar[dr]_-{(x_0,\ldots, x_0)}\ar[r]^-{x_0}&X\ar[d]\ar[dr]\\
&X^r\ar[r]&X^{r+1},
}
\] where the unmarked arrows are the appropriate diagonal maps. For the second claim, we note that the equality $\mathrm{ATC}_r(X)=0$ is equivalent to the existence of a section of $\pi^r_X$. In the case of a singleton, this map is a homeomorphism for every $r>0$, hence certainly admits a section; therefore, by homotopy invariance and the fourth claim, we have $\mathrm{ATC}_r(X)=0$ for every $r>0$ when $X$ (hence $X^r$) is contractible. On the other hand, if $\mathrm{ATC}_r(X)=0$ for some $r>0$, then $\mathrm{acat}(X)=0$ by the fifth claim. Thus, the map $\pi_X$ admits a section, and we conclude that $X$ is a retract of a contractible space, hence contractible.
\end{proof}

This result yields an infinite family of examples for which category and analog category coincide.

\begin{corollary}\label{cor:spheres}
For any $0<d<\infty$, we have $\mathrm{acat}(S^d)=1$.
\end{corollary}
\begin{proof}
The lower bound follows from Proposition \ref{prop:basic}(2). The upper bound follows from Proposition \ref{prop:basic}(3) and the classical calculation $\mathrm{cat}(S^d)=1$.
\end{proof}

\begin{remark}
We record a remark made by a referee. As noted above, a pair of maps $f_i:X_i\to Y$ and a map $g:X_1\to X_2$ with $f_2\circ g=f_1$ induce a commutative diagram 
\[\xymatrix{
\mathcal{P}(f_1)\ar[dr]_-{p_{f_1}}\ar[rr]^-{g_*}&&\mathcal{P}(f_2)\ar[dl]^-{p_{f_2}}\\
&Y
}
\] In particular, it follows that $\mathrm{asecat}(f_1)\geq \mathrm{asecat}(f_2)$. Taking $g$ to be the map $(X^r)^{[0,1]}\to X^{[0,1]}$ given by concatenation, this observation supplies an alternative proof of Proposition \ref{prop:basic}(4).
\end{remark}

We close this section by recording an important submultiplicative law for finite covers. As we will see, this inequality produces strong divergence between our analog invariants and their classical counterparts in many cases.

\begin{proposition}\label{prop:finite cover}
Fix $r>0$. If $p:E\to X$ is a degree $k$ covering map, then \[\mathrm{ATC}_r(X)+1\leq k(\mathrm{ATC}_r(E)+1).\]
\end{proposition}
\begin{proof}
We prove the case $r=1$, the case $r>1$ being entirely analogous. Choose a basepoint $e_0\in p^{-1}(x_0)$, and suppose that $\mathrm{acat}(E)=n$, so that there is a section $s:E\to \mathcal{P}(\pi_E)$ lying in $\mathcal{P}_{n+1}(\pi_E)$. Consider the composite map
\[X\xrightarrow{p^*} \mathcal{P}_k(E)\xrightarrow{s_*} \mathcal{P}_k(\mathcal{P}_{n+1}(\pi_E))\to \mathcal{P}_k(\mathcal{P}_{n+1}(\pi_X))\xrightarrow{\mu}\mathcal{P}_{k(n+1)}(\pi_X)\subseteq \mathcal{P}(\pi_X),\] where the unmarked arrow is induced by the morphism in the arrow category given by the following commutative diagram: 
\[\xymatrix{
(E,e_0)^{([0,1],\{0\})}\ar[d]_-{\pi_E}\ar[r]^-{p\circ(-)}&(X,x_0)^{([0,1],\{0\})}\ar[d]^-{\pi_X}\\
E\ar[r]^-p&X
}\]
A direct calculation reveals that this map is a section of the projection, implying the claim.
\end{proof}

\begin{corollary}\label{cor:projective}
For any $r>0$, we have $\mathrm{ATC}_r(\mathbb{RP}^d)<\mathrm{TC}_r(\mathbb{RP}^d)$ for all but finitely many $d$.
\end{corollary}
\begin{proof}
Combining Corollary \ref{cor:spheres} and Propositions \ref{prop:basic} and \ref{prop:finite cover}, we have
\begin{align*}
\mathrm{ATC}_r(\mathbb{RP}^d)&\leq 2(\mathrm{ATC}_r(S^d)+1)-1\\
&\leq 2(\mathrm{TC}_r(S^d)+1)-1\\
&\leq 2r+1,
\end{align*} which is constant in $d$. On the other hand, we have $\mathrm{TC}_r(\mathbb{RP}^d)\geq \mathrm{cat}(\mathbb{RP}^d)=d$, implying the claim.
\end{proof}

\begin{remark}
At the cost of some arithmetic, Corollary \ref{cor:projective} could be improved by appealing to the lower bounds of \cite{Davis:BHTCRPSIBBP}.
\end{remark}

In the case $r=2$, we have the following.

\begin{proposition}\label{prop:projective}
For any $d>0$, we have $\mathrm{ATC}(\mathbb{RP}^d)=1$.
\end{proposition}
\begin{proof}
Define $\theta:\mathbb{RP}^d\times\mathbb{RP}^d\to \mathbb{R}_{\geq0}$ by setting $\theta(\ell_1,\ell_2)$ to be the absolute value of the dot product between unit vectors in $\ell_1$ and $\ell_2$. The function $\theta$ is well-defined and continuous. Equip $\mathbb{RP}^d$ with the metric induced by the standard metric on $S^d$, and consider the two shortest geodesics from $\ell_1$ to $\ell_2$. Contemplation of geodesics on the sphere shows that these two geodesics have equal length precisely when $\theta(\ell_1,\ell_2)=0$. Away from this locus, we write $\gamma_1(\ell_1,\ell_2)$ and $\gamma_2(\ell_1,\ell_2)$ for the shortest and second shortest geodesics, respectively; on this locus, we assign the labels $1$ and $2$ arbitrarily. Then the function $\frac{1}{2}((1+\theta)\gamma_1+(1-\theta)\gamma_2)$ is a continuous section lying in $\mathcal{P}_2(\pi^2_{\mathbb{RP}^d})$.
\end{proof}

From this calculation and the monotinicity in $r$ of Proposition \ref{prop:basic}(5), we conclude the following.

\begin{corollary}
For any $d>0$, we have $\mathrm{acat}(\mathbb{RP}^d)=1$.
\end{corollary}

\section{Calculations for aspherical spaces}

We turn now to the study of spaces whose fundamental groups dictate their entire homotopy types. In this setting, the numerical invariants $\mathrm{ATC}_r$ become group invariants. More specifically, given a (discrete) group $G$, we write $BG$ for the geometric realization of the nerve of $G$. This space is a CW complex, hence convenient. 

\begin{remark}\label{remark:locally finite}
If $G$ is infinite, the CW complex $BG$ is not locally finite, hence not metrizable. If $G$ is not finitely generated, then the same applies to any homotopy equivalent CW complex. Thus, the definition of \cite{DranishnikovJauhari:DTCLSC} does not apply directly in this setting.
\end{remark}

In the finite case, we have the following general upper bound.

\begin{theorem}\label{thm:finite group}
For any finite group $G$, we have $\mathrm{ATC}_r(BG)\leq |G|-1$ for $r>0$.
\end{theorem}
\begin{proof}
Applying Proposition \ref{prop:finite cover} to the quotient map $EG\to BG$, a degree $|G|$ cover, and invoking Proposition \ref{prop:basic}, we have 
\begin{align*}
\mathrm{ATC}_r(BG)\leq |G|(\mathrm{ATC}_r(EG)+1)-1=|G|-1.
\end{align*} since $EG$ is contractible.
\end{proof}

Since $\mathbb{RP}^\infty\cong BC_2$ is not contractible, Proposition \ref{prop:basic} implies the following calculation.

\begin{corollary}
We have $\mathrm{ATC}_r(\mathbb{RP}^\infty)=1$ for every $r>0$.
\end{corollary}

In contrast, the Eilenberg--Ganea theorem implies that the category of $BG$ is infinite for $G$ finite. In spite of this discrepancy, as the following result asserts, a direct analogue of the Eilenberg--Ganea theorem holds in the torsion-free setting.

\begin{theorem}\label{thm:cd}
For any torsion-free group $G$, we have $\mathrm{acat}(BG)=\mathrm{cd}(G)$. 
\end{theorem}

In light of Corollary \ref{cor:derived} and Remark \ref{remark:diagonal}, this result amounts to calculating the analog sectional category of the covering map $q:EG\to BG$, a fiber bundle with fiber $G$ and structure group $G$. By Example \ref{example:discrete}, Corollary \ref{cor:fiber bundle} and Lemma \ref{lem:bundle comparison}, there is a homeomorphism $\mathcal{P}(q)\cong EG\times_G\Delta^G$ over $BG$, under which the $(n+1)$st stage of the cardinality filtration is identified with $EG\times_G\Delta^G_n$ (here, in a standard notational abuse, we refer to the Borel construction or balanced product, not to a pullback).

\begin{lemma}\label{lem:free action}
The group $G$ is torsion-free if and only if the canonical action of $G$ on $\Delta^G$ is free.
\end{lemma}
\begin{proof}
If $g\in G$ has finite order, then $g$ fixes the barycenter of the face of $\Delta^G$ spanned by the powers of $g$. Conversely, suppose that $g$ has infinite order and fixes $x\in \Delta^G$. The point $x$ lies in the face spanned by some finite subset $\{g_1,\ldots, g_n\}\subseteq G$. Writing $(t_1,\ldots, t_n)$ for its barycentric coordinates in this face, we may assume without loss of generality that $t_i\neq 0$ for $1\leq i\leq n$. After conjugating $g$ if need be, we may further assume that $g_1=1$. Since $g$ is nontrivial and fixes $x$, it follows without loss of generality that $g_2=g$ and $t_1=t_2$. Continuing in this way, and using our assumption that $g$ has infinite order, it follows that $x$ is the barycenter of the simplex spanned by the set $\{1, g, g^2,\ldots, g^{n-1}\}$. Using once more the assumption that $g$ fixes $x$, it follows that $g^n$ lies in this set, a contradiction.
\end{proof}

\begin{proof}[Proof of Theorem \ref{thm:cd}]
As observed above, we wish to calculate the minimal $n$ for which the bundle $EG\times_G\Delta^G_{n}\to BG$ admits a section.

We first show that the bundle in question admits a section when $n=\mathrm{cd}(G)$ using obstruction theory. Having constructed a section over the $k$-skeleton, the case $k=0$ being trivial, the obstruction to extending this section, perhaps after adjustment by a homotopy, lies in $H^{k+1}(BG; {\pi_k(\Delta^G_n)})$. If $k<n$, then $\pi_k(\Delta_n^G)\cong \pi_k(\Delta^G)=0$; and, if $k\geq n=\mathrm{cd}(G)$, then $H^{k+1}(BG;M)=0$ for every local coefficient system $M$. In any case, the group containing the obstruction vanishes, so the section extends to all of $BG$.

Finally, suppose for contradiction that a section exists when $n=\mathrm{cd}(G)-1$, and note that $EG\times_G\Delta_n^G\simeq \Delta_n^G/G$, since the action of $G$ on $\Delta^G$ is free by Lemma \ref{lem:free action}. In the case $n=0$, we have $\Delta^G_0=G$, and it follows that $BG$ is a homotopy retract of a point, hence contractible, so $\mathrm{cd}(G)=0$, a contradiction. In the case $n=1$, it follows that $BG$ is a homotopy retract of a bouquet of circles, whence $G$ is a retract of a free group, hence free; therefore, we have $\mathrm{cd}(G)=1$, a contradiction. In the case $n>1$, choose a $G$-module $M$ such that $H^{n+1}(BG;M)\neq0$. Since $n>1$, the map $EG\times_G\Delta_n^G\to BG$ induces an isomorphism on fundamental groups, so $M$ defines a system of local coefficients on the source as well. A section induces the second arrow in the factorization
\[H^{n+1}(BG;M)\to H^{n+1}(\Delta_n^G/G;M)\to H^{n+1}(BG;M)\]
of the identity map. Since $\Delta_n^G/G$ is a CW complex of dimension $n$, the middle group is trivial, so the source is also trivial, a contradiction.
\end{proof}

\begin{remark}
Only the last paragraph of this proof uses our assumption that $G$ is torsion-free. In the presence of torsion, the argument shows that $EG\times_G\Delta^G\to BG$ admits a section, which is obvious, since this map is a trivial fibration. In the finite case, since $\Delta^G=\Delta^G_{|G|}$, we obtain an alternative argument for Theorem \ref{thm:finite group}.
\end{remark}

For $r>1$, the upper bound still holds. 

\begin{corollary}\label{cor:cd ATC}
For any torsion-free group $G$ and $r>1$, we have $\mathrm{ATC}_r(BG)\leq\mathrm{cd}(G^r)$.
\end{corollary}
\begin{proof}
Since $BG^r\cong B(G^r)$, the claim follows from Proposition \ref{prop:basic} and Theorem \ref{thm:cd}.
\end{proof}

\begin{remark}
This inequality can be strict, even for free products of free Abelian groups, since it is already strict for $\mathrm{TC}_r$ (see \cite{AguilarGuzmanGonzalezOprea:RAAGPPTCGF}, for example).
\end{remark}

We close this section by identifying a potential cohomological avenue for interrogating the inequality of Corollary \ref{cor:cd ATC}. In what follows, we regard $G^r$ as a right $G$-set via the diagonal.

\begin{theorem}\label{thm:ATC cd equality}
Let $G$ be a torsion-free group, fix $r>1$, suppose $\mathrm{cd}(G^r)-1=n>1$, and let $M$ be a $G^r$-module such that $H^{n+1}(BG^r;M)\neq0$. The inequality of Corollary \ref{cor:cd ATC} is an equality provided the differential 
\[d_{n+1}:H^0\left(BG^r; H^n(\Delta_n^{G^r/G}; M)\right)\to H^{n+1}(BG^r;M)\] is nonzero in the Serre spectral sequence for the fibration $EG^r\times_{G^r}\Delta^{G^r/G}_n\to BG^r$.
\end{theorem}
\begin{proof}
Consider the following commutative diagram, in which the unmarked arrows are induced by the diagonal homomorphism $G\to G^r$:
\[\xymatrix{
BG^{[0,1]}\ar[d]_-{\pi_{BG}^r}&BG\ar[l]\ar[d]\ar[rr]&&EG^r/G\ar[d]\ar@{=}[r]^-\sim&EG^r\times_{G^r}G^r/G\ar[d]\\
BG^r\ar@{=}[r]&BG^r\ar@{=}[rr]&&BG^r\ar@{=}[r]&BG^r.
}\] In view of Corollary \ref{cor:derived} and Remark \ref{remark:diagonal}, commutativity implies that it suffices to calculate the analog sectional category of the covering map $q_r:EG^r\times_{G^r}G^r/G\to BG^r$, a fiber bundle with fiber $G^r/G$ and structure group $G^r$. Following the reasoning of Theorem \ref{thm:cd}, then, we wish to show that the fiber bundle $EG^r\times_{G^r}\Delta^{G^r/G}_{n}\to BG^r$ admits no section.

Since $n>1$, the fundamental group of the total space is $G^r$, so we may contemplate the spectral sequence in question. By inspection, it is concentrated in the strips $q=0$ and $q=n$, so $E_{\infty}^{n+1,0}=E_{n+1}^{n+1,0}/\mathrm{im}(d_{n+1})$. Our assumption implies that this quotient is nontrivial, and it follows by naturality of the Serre spectral sequence that the homomorphism $H^{n+1}(BG^r;M)\to H^{n+1}(EG^r\times_{G^r}\Delta^{G^r/G}_n;M)$ induced by the projection is not injective. We conclude that this projection admits no section.
\end{proof}

Regrettably, at present, this result remains a tool in search of an application.

\section{Proof of Theorem \ref{thm:point set}}\label{section:proof of point set}

The space $\mathcal{P}(X)$ is the one of interest for our purposes, but it will be convenient in establishing some of its properties to relate it to a larger space of measures of finite support (which need not be probability measures), defined as the quotient 
\[
\mathcal{M}(X)=\faktor{\bigsqcup_{n\geq0}X^n\times \mathbb{R}_{\geq0}^{n}}{\sim}
\] by the same three relations as in Definition \ref{def:probability measures} (note that we include the zero measure). The discussion in and around Lemma \ref{lem:lowest terms} and Definition \ref{def:support} goes through unchanged for the space $\mathcal{M}(X)$, as does the prolongation to a functor. The inclusions of the various simplices induces a canonical continuous injection $\mathcal{P}(X)\to \mathcal{M}(X)$, which we will show to be a closed embedding in the course of the proof of Theorem \ref{thm:point set}.

Given a measure $\mu$ and $A\subseteq X$, we write $\mu(A)=\sum_{x\in A}\mu(x)$.

\begin{lemma}\label{lem:volume}
The assignment $\mu\mapsto \mu(X)$ defines a continuous map $\mathcal{M}(X)\to \mathbb{R}_{\geq0}$.
\end{lemma}
\begin{proof}
This function is induced by various disjoint unions of composites $X^n\times\mathbb{R}^n_{\geq0}\to \mathbb{R}^n_{\geq0}\to\mathbb{R}_{\geq0}$ of projection and addition of real numbers, each of which is continuous.
\end{proof}

We write $\mathcal{M}_\delta(X)$ for the inverse image of $[0,\delta)\in\mathbb{R}_{\geq0}$ under the map of Lemma \ref{lem:volume}. 

\begin{construction}
For every $m,n\geq0$, we have the canonical homeomorphisms \[X^m\times\mathbb{R}^m_{\geq0}\times X^n\times\mathbb{R}^n_{\geq0}\xrightarrow{\cong} X^{m+n}\times\mathbb{R}^{m+n}_{\geq0}.\] Observing that these maps obviously respect the defining relations of $\mathcal{M}(X)$, we obtain a (a priori discontinuous) function $\mathcal{M}(X)\times\mathcal{M}(X)\to \mathcal{M}(X)$, whose value on $(\mu_1,\mu_2)$ we denote $\mu_1+\mu_2$.
\end{construction}

As shown below in Corollary \ref{cor:monoid}, this map is continuous, but we do not need this fact now.

\begin{construction}\label{construction:basis}
Given $k>0$, tuples $U=(U_1,\ldots, U_k)\in\mathrm{Op}(X)^k$ and $r=(r_1,\ldots, r_k)\in \mathbb{R}^k_{>0}$, and $\epsilon>0$, define $\mathcal{M}(U,r,\epsilon)\subseteq \mathcal{M}(X)$ to be the set of measures $\mu$ such that 
\begin{enumerate}
\item $|\mu(U_j)-r_j|<\epsilon$ for $1\leq j\leq k$, and
\item $\mu(x)=0$ for $x\notin \bigcup_jU_j$.
\end{enumerate}
We further define $\mathcal{M}(U,r,\epsilon,\delta)=\mathcal{M}(U,r,\epsilon)+\mathcal{M}_\delta(X)$.
\end{construction}

\begin{lemma}\label{lem:big basis}
Each of the subsets $\mathcal{M}(U,r,\epsilon,\delta)\subseteq\mathcal{M}(X)$ is open. If $X$ is compact, then the collection of all such forms a basis for the topology of $\mathcal{M}(X)$.
\end{lemma}
\begin{proof}
For the first claim, we are required to verify that $q_m^{-1}\mathcal{M}(U,r,\epsilon,\delta)$ is open for each $m$, where $q_m:X^m\times\mathbb{R}^m_{\geq0}\to \mathcal{M}(X)$ is the projection. Now, if $\mu=\sum_{i=1}^nt_i \delta_{x_i}$, then 
\[
q_m^{-1}(\mu)=\Sigma_m\cdot\left(\bigcup_{m_0+\cdots+m_n=m} X^{m_0}\times \prod_{i=1}^n \Delta_{m_i}(x_i)\times\{0\}^{m_0}\times \prod_{i=1}^n \alpha_{m_i}^{-1}(t_i),
\right)\] where $\Delta_a:X\to X^a$ is the diagonal and $\alpha_a:\mathbb{R}^a_{\geq0}\to\mathbb{R}_{\geq0}$ is addition. Assuming that $\mu\in \mathcal{M}(U,r,\epsilon,\delta)$, we may write $\mu_1=\sum_{i=1}^{n_0}t_i\delta_{x_i}$ and $\mu_2=\sum_{i=n_0+1}^n t_i\delta_{x_i}$ such that $x_i\in\bigcup_jU_j$ for $i\leq n_0$ (without loss of generality), $|\mu_1(U_j)-r_j|<\epsilon$ for each $j$, and $\mu_2(X)<\delta$. For each $i$, let 
\[V_i=\begin{cases}
\bigcap_{j: x_i\in U_j} U_j&\quad \text{if }i\leq n_0\\
X&\quad\text{otherwise.}
\end{cases}
\] For any $\tilde\epsilon,\tilde\delta>0$, the subset $X^{m_0}\times \prod_{i=1}^n \Delta_{m_i}(x_i)\times\{0\}^{m_0}\times \prod_{i=1}^n \alpha_{m_i}^{-1}(t_i)$ is contained in the open neighborhood\[X^{m_0}\times\prod_{i=1}^n V_i^{m_i}\times [0,\tilde\delta)^{m_0}\times\prod_{i=1}^n\alpha^{-1}_{m_i}(t_i-\tilde\epsilon,t_i+\tilde\epsilon),\] and the image of this open neighborhood\footnote{We use that the topology of the convenient product is \emph{finer} than the ordinary product topology.} under $q_m$ lies in $\mathcal{M}(U,r,\epsilon,\delta)$ provided we choose $\tilde\epsilon<\min_j\frac{\epsilon-|\mu_1(U_j)-r_j|}{n_0}$ and $\tilde\delta<\frac{\delta-\mu_2(X)}{m_0+n-n_0}$. The claim follows.

For the second claim, we begin by noting that, since $X$ and $\mathbb{R}_{\geq0}$ are locally compact Hausdorff, the convenient product $X^n\times\mathbb{R}_{\geq0}^n$ carries the ordinary product topology by Proposition \ref{prop:locally compact}. Choose an open subset $V\subseteq \mathcal{M}(X)$ and $\mu\in V$, writing $\mu=\sum_{i=1}^nt_i\delta_{x_i}$ in lowest terms. The point $(x,x_1,\ldots, x_n, 0,t_1,\ldots, t_n)$ lies in $q_{n+1}^{-1}(\mu)$ for every $x\in X$, so there exist open subsets $x_i\in U_i\subseteq X$ and sufficiently small $\epsilon,\delta>0$ such that \[
X\times \prod_{i=1}^nU_i\times [0,\delta)\times\prod_{i=1}^n(t_i-\epsilon, t_i+\epsilon)\subseteq q_{n+1}^{-1}(V)
\] (here we use the tube lemma, hence our assumption on $X$). It follows that $V$ contains $\mathcal{M}(U,t,\epsilon,\delta)$, and the claim follows.
\end{proof}

\begin{corollary}\label{cor:big hausdorff}
The space $\mathcal{M}(X)$ is convenient.
\end{corollary}
\begin{proof}
It suffices by Proposition \ref{prop:quotient} to show that $\mathcal{M}(X)$ is Hausdorff, which follows from Lemma \ref{lem:big basis}, since any two measures can be separated by sets of the form $\mathcal{M}(U,r,\epsilon,\delta)$.
\end{proof}

\begin{corollary}\label{cor:small hausdorff}
The space $\mathcal{P}(X)$ is convenient.
\end{corollary}
\begin{proof}
It suffices by Proposition \ref{prop:quotient} to show that $\mathcal{P}(X)$ is Hausdorff, which is immediate from Corollary \ref{cor:big hausdorff} and the existence of the continuous injection $\mathcal{P}(X)\to \mathcal{M}(X)$.
\end{proof}

We also record the following fact, which we do not use.

\begin{corollary}\label{cor:monoid}
Addition of measures equips $\mathcal{M}(X)$ with the structure of a unital Abelian topological monoid.
\end{corollary}
\begin{proof}
The main point is that, since $\mathcal{M}(X)$ is convenient by Corollary \ref{cor:big hausdorff}, addition of measures is continuous by Proposition \ref{prop:quotient product}.
\end{proof}

\begin{proof}[Proof of Theorem \ref{thm:point set}]
We have already established the claim of convenience. Given a homotopy $H:[0,1]\times X\to Y$ between $f$ and $g$, Proposition \ref{prop:quotient product} implies that the composites
\[
[0,1]\times X^n\times\Delta^{n-1}\xrightarrow{\Delta\times \id} [0,1]^n\times X^n\times\Delta^{n-1}\xrightarrow{H^n\times \id} Y^n\times\Delta^{n-1}
\] descend to a homotopy between $f_*$ and $g_*$.

Since colimits commute, and since convenient colimits are ordinary colimits, the claim regarding sifted colimits is a standard consequence of the fact that the convenient product distributes over colimits by Proposition \ref{prop:cartesian}. For the claim regarding quotient maps, suppose that $f:X\to Y$ is a quotient map, and consider the following commutative diagram:
\[
\xymatrix{
\displaystyle\bigsqcup_{n>0}X^n\times\Delta^{n-1}\ar[d]\ar[r]&\displaystyle\bigsqcup_{n>0}Y^n\times\Delta^{n-1}\ar[d]\\
\mathcal{P}(X)\ar[r]^-{f_*}&\mathcal{P}(Y).
}
\] The vertical maps are quotient maps by definition, and the top horizontal map is a quotient map by Proposition \ref{prop:quotient product}. It follows that the counterclockwise composite is a quotient map, then that $f_*$ is so. 

For the claim regarding embeddings, we first establish the corresponding claim for $\mathcal{M}$. Given a closed subspace $A\subseteq X$, note first that $A$ is again convenient, so $\mathcal{M}(A)$ is defined. Second, the induced map $\mathcal{M}(A)\to \mathcal{M}(X)$ is injective. Third, the image of this map is closed, since the measure $\mu$ lying outside the image is separated from it by the open set $\mathcal{M}(X\setminus A,\mu(X\setminus A), \epsilon,\delta)$ provided $\epsilon<\mu(X\setminus A)$. It remains to show that the map is a topological embedding.

Assume first that $A$ is compact. Given a basis element $\mathcal{M}(U,r,\epsilon,\delta)$ for $\mathcal{M}(A)$ as in Lemma \ref{lem:big basis}, choose open subsets $\widetilde U_i\subseteq X$ with $\widetilde U_i\cap A=U_i$. Then the intersection of $\mathcal{M}(\widetilde U,t,\epsilon,\delta)$ with the image of $\mathcal{M}(A)$ is precisely the image of $\mathcal{M}(U,t,\epsilon,\delta)$. It follows that the natural map $\mathcal{M}(A)\to \mathcal{M}(X)$ is an embedding.

In the general case, writing $\mathcal{K}$ for the collection of compact subsets of $X$, partially ordered by inclusion, consider the commutative diagram
\[\xymatrix{
\colim_{K\in\mathcal{K}}\mathcal{M}(K\cap A)\ar[d]\ar[r]&\colim_{K\in\mathcal{K}}\mathcal{M}(K)\ar[d]\\
\mathcal{M}(\colim_{K\in\mathcal{K}} K\cap A)\ar[d]\ar[r]&\mathcal{M}(\colim_{K\in\mathcal{K}} K)\ar[d]\\
\mathcal{M}(A)\ar[r]&\mathcal{M}(X).
}\] Since $\mathcal{K}$ is a filtered, hence sifted, category, the top two vertical arrows are homeomorphisms, and the bottom two vertical arrows are homeomorphisms by convenience. Thus, it will suffice to show, given a closed set $i:C\subseteq \mathcal{M}(A)$ and compact $j:K\subseteq X$, that $j_*^{-1}i_*(C)$ is closed. By the previous case, the top horizontal arrow in the commutative diagram
\[
\xymatrix{
\mathcal{M}(K\cap A)\ar[r]^-{\tilde i_*}\ar[d]_-{\tilde j_*}&\mathcal{M}(K)\ar[d]^-{j_*}\\
\mathcal{M}(A)\ar[r]^{i_*}&\mathcal{M}(X)
}
\] is a closed embedding, the remaining arrows are injective, and the diagram is a pullback at the level of sets, so $j_*^{-1}i_*(C)=\tilde i_*(\tilde j_*)^{-1}(C)$ is closed, as desired.

To conclude, it suffices to establish that the canonical map $\mathcal{P}(X)\to \mathcal{M}(X)$ is a closed embedding. By the same colimit-and-pullback argument as above, we may assume that $X$ is compact. In this case, each of the restrictions $\mathcal{P}_n(X)\to \mathcal{M}_n(X)$ is a closed embedding, as it is a continuous injection with compact source and Hausdorff target. The claim then follows by another colimit-and-pullback argument, where the colimit is taken over $n$.
\end{proof}

\begin{appendix}

\section{Convenient topology}\label{appendix:convenient}

Here we review the point set topological results from \cite{Steenrod:CCTS} we require.

\begin{definition}
We say that a topological space $X$ is \emph{convenient} if it is Hausdorff and carries the weak topology with respect to its compact subsets.
\end{definition}

\begin{remark}
Convenient spaces are the same thing as compactly generated spaces in the sense of \cite{Steenrod:CCTS}. We avoid this terminology, as it has become fraught with ambiguity through other usage.
\end{remark}

According to \cite[3.2]{Steenrod:CCTS}, convenient spaces form a coreflective subcategory of the category of all Hausdorff spaces; in particular, convenient colimits are ordinary colimits. The value of the right adjoint $k$ on a Hausdorff space $X$ is the same set endowed with the weak topology with respect to the compact subsets of $X$. Since $k$ is a right adjoint and $k^2$ is naturally isomorphic to the identity, convenient limits are formed by applying $k$ to the corresponding ordinary limit.

\begin{proposition}[{\cite[2.6]{Steenrod:CCTS}}]\label{prop:quotient}
Let $q:X\to Y$ be a quotient map. If $X$ is convenient and $Y$ Hausdorff, then $Y$ is convenient.
\end{proposition}

\begin{proposition}[{\cite[4.4]{Steenrod:CCTS}}]\label{prop:quotient product}
The collection of quotient maps between convenient spaces is closed under finite products.
\end{proposition}

\begin{proposition}[{\cite[4.3]{Steenrod:CCTS}}]\label{prop:locally compact}
If $X$ is locally compact Hausdorff and $Y$ convenient, then the topology on the convenient product $X\times Y$ is the ordinary product topology.
\end{proposition}

\begin{proposition}[{\cite[5.6]{Steenrod:CCTS}}]\label{prop:cartesian}
The category of convenient topological spaces is Cartesian closed.
\end{proposition}

\begin{proposition}\label{prop:closed embedding}
The functor $k$ preserves closed embeddings.
\end{proposition}
\begin{proof}
Without loss of generality, we may consider the inclusion $i:A\to X$ of a closed subspace of a (not necessarily convenient) topological space $X$. Since $A$ is closed, the intersection of a compact subset of $K$ with $A$ is again compact; therefore, a subset $C\subseteq A$ that is compactly closed in $A$ is also compactly closed in $X$, which is to say that $k(i)$ is a closed map. Since $k(i)$ remains injective, the claim follows.
\end{proof}

\begin{corollary}\label{cor:pullback inclusion}
Given maps $f_i:X_i\to Y$ between convenient spaces for $i\in \{1,2\}$, the canonical map $X_1\times_Y X_2\to X_1\times X_2$ is a closed embedding.
\end{corollary}
\begin{proof}
As a right adjoint, the functor $k$ preserves limits, so the pullback and product in question are obtained by forming the respective ordinary limits in the category of topological spaces and applying $k$. By Proposition \ref{prop:closed embedding}, then, it suffices to establish the corresponding claim for the map between ordinary limits. For this purpose, we note that the pullback is the preimage under $f_1\times f_2$ of the diagonal in $Y$, which is closed since $Y$ is Hausdorff.
\end{proof}

\end{appendix}

\bibliographystyle{amsalpha}
\bibliography{references}

\end{document}